\documentclass{article}
\usepackage{float, hyperref, amsmath, amssymb, MnSymbol, amsfonts, enumerate, vmargin, wasysym, accents, multicol, tikz, spverbatim, cjhebrew, amsthm, relsize, framed, listings, graphicx,
caption}

\usetikzlibrary{decorations.markings,decorations.pathreplacing,calc,arrows}
\usetikzlibrary{calc}

\hypersetup{colorlinks=true, linkcolor=violet, filecolor=blue, citecolor=violet, urlcolor=darkgray, citebordercolor=lightgray, filebordercolor=lightgray, linkbordercolor=lightgray}

\usepackage{chngcntr}
\counterwithout{paragraph}{subsubsection} 
\counterwithin{paragraph}{section}

\setmarginsrb{30mm}{20mm}{30mm}{20mm}{0pt}{10mm}{0pt}{10mm}

\def\N{\mathbb{N}}  \def\Z{\mathbb{Z}} \def\R{\mathbb{R}} 

\newtheorem{thm}{Theorem} \newtheorem*{thm*}{Theorem}
 \newtheorem*{claim*}{Claim}
 \newtheorem*{dfn*}{Definition}
\newtheorem{lemma}[thm]{Lemma} \newtheorem*{lemma*}{Lemma}
 \newtheorem*{prop*}{Proposition}
\newtheorem{cor}[thm]{Corollary} \newtheorem*{cor*}{Corollary}
 \newtheorem*{conj*}{Conjecture}
 \newtheorem*{quest*}{Question}
 \newtheorem*{exer*}{Exercise}
 \newtheorem*{example*}{Example}
\theoremstyle{remark}
 \newtheorem*{rmk*}{Remark}
 \newtheorem*{examples*}{Examples}
 \newtheorem*{rmks*}{Remarks}
 \newtheorem*{disc*}{Discussion}
\newtheorem*{discq*}{Discussion and Questions}

\let\mc\multicolumn

\newcommand{\switch}
{
\tikz[thick,scale=1]{
\draw[gray, line width = 5] (-2.5,0) -- (2.5,0);
\draw[black, line width = 3, postaction={decorate}, decoration={markings, mark=at position 0.85 with {\arrow{>}}}] (-1.5,-1.5) -- (-1.5,1.5);
\draw[black, line width = 3, dotted] (-1.5,-2) -- (-1.5,-1.5);
\draw[black, line width = 3, dotted] (-1.5,2) -- (-1.5,1.5);
\node[draw,fill=white,circle,line width = 2] at (-1.5,-1.25) {1};
\draw[black, line width = 3, postaction={decorate}, decoration={markings, mark=at position 0.85 with {\arrow{>}}}] (0,1.5) -- (0,-1.5);
\draw[black, line width = 3, dotted] (0,-2) -- (0,-1.5);
\draw[black, line width = 3, dotted] (0,2) -- (0,1.5);
\node[draw,fill=white,circle,line width = 2] at (0,1.25) {2};
\draw[black, line width = 3, postaction={decorate}, decoration={markings, mark=at position 0.85 with {\arrow{>}}}] (1.5,-1.5) -- (1.5,1.5);
\draw[black, line width = 3, dotted] (1.5,-2) -- (1.5,-1.5);
\draw[black, line width = 3, dotted] (1.5,2) -- (1.5,1.5);
\node[draw,fill=white,circle,line width = 2] at (1.5,-1.25) {3};
\node at (-2.25,0.25) {E};
\node at (-1.80,1.50) {C};
\pgfresetboundingbox \clip (-2.5,-2.4) rectangle (2.5,2.4);}
\tikz[thick,scale=1]{
\node[scale=3] at (0,0) {$\leadsto$};
\pgfresetboundingbox \clip (-1.5,-2.4) rectangle (1.5,2.4);}
\tikz[thick,scale=1]{
\draw[gray, line width = 5] (-2.5,0) -- (2.5,0);
\draw[black, line width = 3, postaction={decorate}, decoration={markings, mark=at position 0.85 with {\arrow{>}}}] (-1.5,-1.5) .. controls +(90:1.5) and +(90:1.5) .. (0,-1.5);
\draw[black, line width = 3, dotted] (-1.5,-2) -- (-1.5,-1.5);
\draw[black, line width = 3, dotted] (-1.5,2) -- (-1.5,1.5);
\draw[black, line width = 3, postaction={decorate}, decoration={markings, mark=at position 0.85 with {\arrow{>}}}] (1.5,-1.5) .. controls +(90:1.5) and +(-90:1.5) .. (-1.5,1.5);
\draw[black, line width = 3, dotted] (0,-2) -- (0,-1.5);
\draw[black, line width = 3, dotted] (0,2) -- (0,1.5);
\draw[black, line width = 3, postaction={decorate}, decoration={markings, mark=at position 0.85 with {\arrow{>}}}] (0,1.5) .. controls +(-90:1.5) and +(-90:1.5) .. (1.5,1.5);
\draw[black, line width = 3, dotted] (1.5,-2) -- (1.5,-1.5);
\draw[black, line width = 3, dotted] (1.5,2) -- (1.5,1.5);
\node[draw,fill=white,circle,line width = 2] at (-1.5,-1.25) {1};
\node[draw,fill=white,circle,line width = 2] at (1.5,-1.25) {2};
\node[draw,fill=white,circle,line width = 2] at (0,1.25) {3};
\node at (-2.25,0.25) {E};
\pgfresetboundingbox \clip (-2.5,-2.4) rectangle (2.5,2.4);}
}

\newcommand{\diagrams}{
\begin{tabular}{ccc}
\begin{tikzpicture}[thick]
\draw [white,line width=2.0pt,double=black,double distance=2.0pt] (2,4) to[out=180,in=270] (1,5);
\draw [white,line width=2.0pt,double=black,double distance=2.0pt] (2,5) to[out=315,in=90] (1,3);
\draw [white,line width=2.0pt,double=black,double distance=2.0pt] (2,5) to[out=135,in=90] (0,4);
\draw [white,line width=2.0pt,double=black,double distance=2.0pt] (3,4) to[out=90,in=180] (4,5);
\draw [white,line width=2.0pt,double=black,double distance=2.0pt] (5,3) to[out=90,in=0] (4,5);
\draw [white,line width=2.0pt,double=black,double distance=2.0pt] (2,2) to[out=0,in=180] (3,2);
\draw [white,line width=2.0pt,double=black,double distance=2.0pt] (4,3) to[out=270,in=0] (3,2);
\draw [white,line width=2.0pt,double=black,double distance=2.0pt] (3,4) to[out=270,in=90] (3,2);
\draw [white,line width=2.0pt,double=black,double distance=2.0pt] (2,3) to[out=0,in=270] (4,5);
\draw [white,line width=2.0pt,double=black,double distance=2.0pt] (4,3) to[out=90,in=0] (2,4);
\draw [white,line width=2.0pt,double=black,double distance=2.0pt] (4,5) to[out=90,in=90] (1,5);
\draw [white,line width=2.0pt,double=black,double distance=2.0pt] (2,0) to[out=0,in=270] (3,2);
\draw [white,line width=2.0pt,double=black,double distance=2.0pt] (2,0) to[out=180,in=270] (1,1);
\draw [white,line width=2.0pt,double=black,double distance=2.0pt] (3,1) to[out=0,in=270] (5,3);
\draw [white,line width=2.0pt,double=black,double distance=2.0pt] (3,1) to[out=180,in=270] (1,3);
\draw [white,line width=2.0pt,double=black,double distance=2.0pt] (2,3) to[out=180,in=270] (0,4);
\draw [white,line width=2.0pt,double=black,double distance=2.0pt] (2,2) to[out=180,in=90] (1,1);
\draw [black,-,line width = 2,decoration={markings,mark=at position 0.95 with {\arrow{>}}},postaction=decorate] (2,2) to[out=180,in=90] (1,1);
\filldraw [black] (2,2) circle [radius=0.1]; \node [right,anchor=south west] at (2,2) {B}; 
\node [anchor=south west] at (4,5) {A}; \node [anchor=south east] at (3,4) {A}; \node [anchor=west] at (3,3) {A}; \node [anchor=north west] at (3,2) {D}; \node [anchor=north west] at (3,1) {A}; \node [anchor=north] at (1,2) {D}; \node [anchor=north] at (4,4) {A}; \node [anchor=north east] at (2,4) {D}; \node [anchor=north east] at (1,5) {A}; \node [anchor=north east] at (1,3) {A};
\pgfresetboundingbox \clip (-0.5,-0.1) rectangle (5.5,6.0);
\end{tikzpicture}
& 
&
\begin{tikzpicture}[thick]
\draw [white,line width=2.0pt,double=black,double distance=2.0pt] (2,4) to[out=180,in=270] (1,5);
\draw [white,line width=2.0pt,double=black,double distance=2.0pt] (2,5) to[out=315,in=90] (1,3);
\draw [white,line width=2.0pt,double=black,double distance=2.0pt] (2,5) to[out=135,in=90] (0,4);
\draw [white,line width=2.0pt,double=black,double distance=2.0pt] (3,4) to[out=90,in=180] (4,5);
\draw [white,line width=2.0pt,double=black,double distance=2.0pt] (5,3) to[out=90,in=0] (4,5);
\draw [white,line width=2.0pt,double=black,double distance=2.0pt] (2,2) to[out=0,in=180] (3,2);
\draw [white,line width=2.0pt,double=black,double distance=2.0pt] (4,3) to[out=270,in=0] (3,2);
\draw [white,line width=2.0pt,double=black,double distance=2.0pt] (3,4) to[out=270,in=90] (3,2);
\draw [white,line width=2.0pt,double=black,double distance=2.0pt] (2,3) to[out=0,in=270] (4,5);
\draw [white,line width=2.0pt,double=black,double distance=2.0pt] (4,3) to[out=90,in=0] (2,4);
\draw [white,line width=2.0pt,double=black,double distance=2.0pt] (4,5) to[out=90,in=90] (1,5);
\draw [white,line width=2.0pt,double=black,double distance=2.0pt] (2,0) to[out=0,in=270] (3,2);
\draw [white,line width=2.0pt,double=black,double distance=2.0pt] (2,0) to[out=180,in=270] (1,1);
\draw [white,line width=2.0pt,double=black,double distance=2.0pt] (3,1) to[out=0,in=270] (5,3);
\draw [white,line width=2.0pt,double=black,double distance=2.0pt] (3,1) to[out=180,in=270] (1,3);
\draw [white,line width=2.0pt,double=black,double distance=2.0pt] (2,3) to[out=180,in=270] (0,4);
\draw [white,line width=2.0pt,double=black,double distance=2.0pt] (2,2) to[out=180,in=90] (1,1);
\draw [black,-,line width = 2,decoration={markings,mark=at position 0.95 with {\arrow{>}}},postaction=decorate] (2,2) to[out=180,in=90] (1,1);
\filldraw [black] (2,2) circle [radius=0.1]; \node [right,anchor=south west] at (2,2) {B}; 
\node [anchor=south west] at (4,5) {A}; \node [anchor=south east] at (3,4) {A}; \node [anchor=west] at (3,3) {A}; \node [anchor=north west] at (3,2) {D}; \node [anchor=north west] at (3,1) {A}; \node [anchor=north] at (1,2) {D}; \node [anchor=north] at (4,4) {A}; \node [anchor=north east] at (2,4) {D}; \node [anchor=north east] at (1,5) {A}; \node [anchor=north east] at (1,3) {A};
\node [circle,draw,gray,line width=2,scale=0.5] (1) at (2.5,1.5) {};
\node [circle,draw,gray,line width=2,scale=0.5] (2) at (1.5,0.5) {};
\node [circle,draw,gray,line width=2,scale=0.5] (3) at (4,2) {};
\node [circle,draw,gray,line width=2,scale=0.5] (4) at (3.5,2.5) {};
\node [circle,draw,gray,line width=2,scale=0.5] (5) at (1.5,2.5) {};
\node [circle,draw,gray,line width=2,scale=0.5] (6) at (3.2,3.65) {};
\node [circle,draw,gray,line width=2,scale=0.5] (7) at (3.5,4.5) {};
\node [circle,draw,gray,line width=2,scale=0.5] (8) at (2.25,3.5) {};
\node [circle,draw,gray,line width=2,scale=0.5] (9) at (0.5,4) {};
\node [circle,draw,gray,line width=2,scale=0.5] (10) at (1.5,5) {};
\node [circle,draw,gray,line width=2,scale=0.5] (11) at (2.5,5.5) {};
\node [circle,draw,gray,line width=2,scale=0.5] (12) at (0.5,2) {};
\path[-,lightgray,line width=3] (4) edge (5);
\path[-,lightgray,line width=3] (10) edge (11);
\path[-,lightgray,line width=3] (10) edge (9);
\path[-,lightgray,line width=3] (8) edge[bend left = 30] (9);
\path[-,lightgray,line width=3] (11) edge (8);
\path[dotted,lightgray,line width=3] (11) edge (7);
\path[-,lightgray,line width=3] (4) edge[bend left = 30] (3);
\path[-,lightgray,line width=3] (12) edge[out=270,in=150,looseness=1.25] (2);
\path[-,lightgray,line width=3] (1) edge (2);
\path[-,lightgray,line width=3] (1) edge[out=0,in=240] (3);
\path[-,lightgray,line width=3] (12) edge (5);
\path[densely dashed,lightgray,line width=3] (5) edge (1);
\path[dotted,lightgray,line width=3] (6) edge (7);
\path[dotted,lightgray,line width=3] (5) edge (8);
\pgfresetboundingbox \clip (-0.5,-0.1) rectangle (5.5,6.0);
\end{tikzpicture} 
\\ 
\mc1l{(A) Assigning $A$/$D$ types to $n=10$ crossings} & & 
\mc1l{(B) The dual subgraph $G$ with $m=5$ [solid],} \\ 
\mc1l{with respect to the base point $B$.}  & & 
\mc1l{an edge through $B$, extension to $T$ [dashed].} \\
&
&
\\
\begin{tikzpicture}[thick]
\draw [white,line width=2.0pt,double=black,double distance=2.0pt] (2,4) to[out=180,in=270] (1,5);
\draw [white,line width=2.0pt,double=black,double distance=2.0pt] (2,5) to[out=315,in=90] (1,3);
\draw [white,line width=2.0pt,double=black,double distance=2.0pt] (2,5) to[out=135,in=90] (0,4);
\draw [white,line width=2.0pt,double=black,double distance=2.0pt] (3,4) to[out=90,in=180] (4,5);
\draw [white,line width=2.0pt,double=black,double distance=2.0pt] (5,3) to[out=90,in=0] (4,5);
\draw [white,line width=2.0pt,double=black,double distance=2.0pt] (2,2) to[out=0,in=180] (3,2);
\draw [white,line width=2.0pt,double=black,double distance=2.0pt] (4,3) to[out=270,in=0] (3,2);
\draw [white,line width=2.0pt,double=black,double distance=2.0pt] (3,4) to[out=270,in=90] (3,2);
\draw [white,line width=2.0pt,double=black,double distance=2.0pt] (2,3) to[out=0,in=270] (4,5);
\draw [white,line width=2.0pt,double=black,double distance=2.0pt] (4,3) to[out=90,in=0] (2,4);
\draw [white,line width=2.0pt,double=black,double distance=2.0pt] (4,5) to[out=90,in=90] (1,5);
\draw [white,line width=2.0pt,double=black,double distance=2.0pt] (2,0) to[out=0,in=270] (3,2);
\draw [white,line width=2.0pt,double=black,double distance=2.0pt] (2,0) to[out=180,in=270] (1,1);
\draw [white,line width=2.0pt,double=black,double distance=2.0pt] (3,1) to[out=0,in=270] (5,3);
\draw [white,line width=2.0pt,double=black,double distance=2.0pt] (3,1) to[out=180,in=270] (1,3);
\draw [white,line width=2.0pt,double=black,double distance=2.0pt] (2,3) to[out=180,in=270] (0,4);
\draw [white,line width=2.0pt,double=black,double distance=2.0pt] (2,2) to[out=180,in=90] (1,1);
\draw [black,-,line width = 2,decoration={markings,mark=at position 0.95 with {\arrow{>}}},postaction=decorate] (2,2) to[out=180,in=90] (1,1);
\filldraw [black] (2,2) circle [radius=0.1]; \node [right,anchor=south west] at (2,2) {B}; 
\node [anchor=south west] at (4,5) {A}; \node [anchor=south east] at (3,4) {A}; \node [anchor=west] at (3,3) {A}; \node [anchor=north west] at (3,2) {D}; \node [anchor=north west] at (3,1) {A}; \node [anchor=north] at (1,2) {D}; \node [anchor=north] at (4,4) {A}; \node [anchor=north east] at (2,4) {D}; \node [anchor=north east] at (1,5) {A}; \node [anchor=north east] at (1,3) {A};
\draw [gray,line width=2] 
(2.0,1.8) -- (2.0,2.3)
(1.3,2.5) -- (0.9,2.2)
(0.8,0.7) -- (1.3,0.7)
(2.3,0.9) -- (2.6,1.3)
(2.7,1.4) -- (3.2,1.4)
(3.7,2.0) -- (3.4,2.4)
(3.2,2.7) -- (2.7,2.7)
(2.4,2.8) -- (2.4,3.3)
(2.4,3.7) -- (2.4,4.2)
(2.8,4.6) -- (3.2,4.3)
(3.3,4.1) -- (3.2,3.7)
(3.3,3.7) -- (3.5,4.1)
(3.5,4.6) -- (3.2,5.0)
(2.2,5.2) -- (1.8,4.9)
(1.4,4.6) -- (1.0,4.3)
(1.0,3.7) -- (1.4,3.4)
(2.0,3.2) -- (2.0,2.7)
(2.6,2.3) -- (2.6,1.8);
\pgfresetboundingbox \clip (-0.5,-0.1) rectangle (5.5,6.3);
\end{tikzpicture}
& 
&
\begin{tikzpicture}[thick]
\draw [white,line width=2.0pt,double=black,double distance=2.0pt] (2,4) to[out=180,in=270] (1,5);
\draw [white,line width=2.0pt,double=black,double distance=2.0pt] (2,5) to[out=315,in=90] (1,3);
\draw [white,line width=2.0pt,double=black,double distance=2.0pt] (2,5) to[out=135,in=90] (0,4);
\draw [white,line width=2.0pt,double=black,double distance=2.0pt] (3,4) to[out=90,in=180] (4,5);
\draw [white,line width=2.0pt,double=black,double distance=2.0pt] (5,3) to[out=90,in=0] (4,5);
\draw [white,line width=2.0pt,double=black,double distance=2.0pt] (2,2) to[out=0,in=180] (3,2);
\draw [white,line width=2.0pt,double=black,double distance=2.0pt] (4,3) to[out=270,in=0] (3,2);
\draw [white,line width=2.0pt,double=black,double distance=2.0pt] (3,4) to[out=270,in=90] (3,2);
\draw [white,line width=2.0pt,double=black,double distance=2.0pt] (2,3) to[out=0,in=270] (4,5);
\draw [white,line width=2.0pt,double=black,double distance=2.0pt] (4,3) to[out=90,in=0] (2,4);
\draw [white,line width=2.0pt,double=black,double distance=2.0pt] (4,5) to[out=90,in=90] (1,5);
\draw [white,line width=2.0pt,double=black,double distance=2.0pt] (2,0) to[out=0,in=270] (3,2);
\draw [white,line width=2.0pt,double=black,double distance=2.0pt] (2,0) to[out=180,in=270] (1,1);
\draw [white,line width=2.0pt,double=black,double distance=2.0pt] (3,1) to[out=0,in=270] (5,3);
\draw [white,line width=2.0pt,double=black,double distance=2.0pt] (3,1) to[out=180,in=270] (1,3);
\draw [white,line width=2.0pt,double=black,double distance=2.0pt] (2,3) to[out=180,in=270] (0,4);
\draw [white,line width=2.0pt,double=black,double distance=2.0pt] (2,2) to[out=180,in=90] (1,1);
\draw [black,-,line width = 2,decoration={markings,mark=at position 0.95 with {\arrow{>}}},postaction=decorate] (2,2) to[out=180,in=90] (1,1);
\filldraw [black] (2,2) circle [radius=0.1]; \node [right,anchor=south west] at (2,2) {B}; 
\node [anchor=south west] at (4,5) {A}; \node [anchor=south east] at (3,4) {A}; \node [anchor=west] at (3,3) {A}; \node [anchor=north west] at (3,2) {D}; \node [anchor=north west] at (3,1) {A}; \node [anchor=north] at (1,2) {D}; \node [anchor=north] at (4,4) {A}; \node [anchor=north east] at (2,4) {D}; \node [anchor=north east] at (1,5) {A}; \node [anchor=north east] at (1,3) {A};
\draw [lightgray,line width=2
] 
(2.0,1.8) -- (2.0,2.3) to[out=90,in=30] 
(1.3,2.5) -- (0.9,2.2) to[out=210,in=180]
(0.8,0.7) -- (1.3,0.7) to[out=0,in=240]
(2.3,0.9) -- (2.6,1.2) to[out=60,in=270] (2.0,1.8);
\draw [lightgray,line width=2
] 
(2.8,1.4) -- (3.2,1.4) to[out=0,in=-60]
(3.7,2.0) -- (3.4,2.4) to[out=120,in=0]
(3.2,2.7) -- (2.7,2.7) to[out=180,in=90]
(2.6,2.3) -- (2.6,1.8) to[out=-90,in=180] (2.8,1.4);
\draw [lightgray,line width=2
] 
(2.4,2.8) -- (2.4,3.3) to[in=90,out=90]
(2.0,3.2) -- (2.0,2.7) to[in=-90,out=-90] (2.4,2.8);
\draw [lightgray,line width=2
] 
(3.3,4.1) -- (3.2,3.7) to[out=270,in=240]
(3.3,3.7) -- (3.5,4.1) to[out=60,in=60] (3.3,4.1);
\draw [lightgray,line width=2
] 
(2.8,4.6) -- (3.2,4.3) to[out=-30,in=-60]
(3.5,4.6) -- (3.2,5.0) to[out=120,in=150] (2.8,4.6);
\draw [lightgray,line width=2
] 
(2.4,3.9) -- (2.4,4.2) to[out=90,in=30]
(2.2,5.2) -- (1.8,4.9) --
(1.4,4.6) -- (1.0,4.3) to[out=210,in=150]
(1.0,3.7) -- (1.4,3.4) to[out=-30,in=-90] (2.4,3.9);
\pgfresetboundingbox \clip (-0.5,-0.1) rectangle (5.5,6.3);
\end{tikzpicture} 
\\ 
\mc1l{(C) Putting $18$ line segments by $G \cup T$,}
&& 
\mc1l{(D) Matching adjacent segment tips in each}
\\
\mc1l{$10 \cdot 1$ for $A$-$D$ edges, $4 \cdot 2$ for $A$-$A$ / $D$-$D$.}
&& 
\mc1l{face yields a circle set.}
\\ 
&
&
\\
\begin{tikzpicture}[thick]
\draw [white,line width=2.0pt,double=black,double distance=2.0pt] (2,4) to[out=180,in=270] (1,5);
\draw [white,line width=2.0pt,double=black,double distance=2.0pt] (2,5) to[out=315,in=90] (1,3);
\draw [white,line width=2.0pt,double=black,double distance=2.0pt] (2,5) to[out=135,in=90] (0,4);
\draw [white,line width=2.0pt,double=black,double distance=2.0pt] (3,4) to[out=90,in=180] (4,5);
\draw [white,line width=2.0pt,double=black,double distance=2.0pt] (5,3) to[out=90,in=0] (4,5);
\draw [white,line width=2.0pt,double=black,double distance=2.0pt] (2,2) to[out=0,in=180] (3,2);
\draw [white,line width=2.0pt,double=black,double distance=2.0pt] (4,3) to[out=270,in=0] (3,2);
\draw [white,line width=2.0pt,double=black,double distance=2.0pt] (3,4) to[out=270,in=90] (3,2);
\draw [white,line width=2.0pt,double=black,double distance=2.0pt] (2,3) to[out=0,in=270] (4,5);
\draw [white,line width=2.0pt,double=black,double distance=2.0pt] (4,3) to[out=90,in=0] (2,4);
\draw [white,line width=2.0pt,double=black,double distance=2.0pt] (4,5) to[out=90,in=90] (1,5);
\draw [white,line width=2.0pt,double=black,double distance=2.0pt] (2,0) to[out=0,in=270] (3,2);
\draw [white,line width=2.0pt,double=black,double distance=2.0pt] (2,0) to[out=180,in=270] (1,1);
\draw [white,line width=2.0pt,double=black,double distance=2.0pt] (3,1) to[out=0,in=270] (5,3);
\draw [white,line width=2.0pt,double=black,double distance=2.0pt] (3,1) to[out=180,in=270] (1,3);
\draw [white,line width=2.0pt,double=black,double distance=2.0pt] (2,3) to[out=180,in=270] (0,4);
\draw [white,line width=2.0pt,double=black,double distance=2.0pt] (2,2) to[out=180,in=90] (1,1);
\draw [black,-,line width = 2,decoration={markings,mark=at position 0.95 with {\arrow{>}}},postaction=decorate] (2,2) to[out=180,in=90] (1,1);
\filldraw [black] (2,2) circle [radius=0.1]; \node [right,anchor=south west] at (2,2) {B}; 
\node [anchor=south west] at (4,5) {A}; \node [anchor=south east] at (3,4) {A}; \node [anchor=west] at (3,3) {A}; \node [anchor=north west] at (3,2) {D}; \node [anchor=north west] at (3,1) {A}; \node [anchor=north] at (1,2) {D}; \node [anchor=north] at (4,4) {A}; \node [anchor=north east] at (2,4) {D}; \node [anchor=north east] at (1,5) {A}; \node [anchor=north east] at (1,3) {A};
\draw [lightgray,line width=2
] 
(2.0,1.8) -- (2.0,2.3) to[out=90,in=30] 
(1.3,2.5) -- (0.9,2.2) to[out=210,in=180]
(0.8,0.7) -- (1.3,0.7) to[out=0,in=240]
(2.3,0.9) -- (2.6,1.3) to[out=60,in=180]
(2.7,1.4) -- (3.2,1.4) to[out=0,in=-60]
(3.7,2.0) -- (3.4,2.4) to[out=120,in=0]
(3.2,2.7) -- (2.7,2.7) to[out=180,in=270]
(2.4,2.8) -- (2.4,3.3) --
(2.4,3.7) -- (2.4,4.2) to[out=90,in=180]
(2.8,4.6) -- (3.2,4.3) to[out=-30,in=60]
(3.3,4.1) -- (3.2,3.7) to[out=270,in=240]
(3.3,3.7) -- (3.5,4.1) to[out=60,in=-60]
(3.5,4.6) -- (3.2,5.0) to[out=120,in=30]
(2.2,5.2) -- (1.8,4.9) --
(1.4,4.6) -- (1.0,4.3) to[out=210,in=150]
(1.0,3.7) -- (1.4,3.4) to[out=-30,in=90]
(2.0,3.2) -- (2.0,2.7) to[out=-90,in=90]
(2.6,2.3) -- (2.6,1.8) to[out=-90,in=-90] (2.0,1.8);
\pgfresetboundingbox \clip (-0.5,-0.1) rectangle (5.5,6.3);
\end{tikzpicture} 
& 
&
\begin{tikzpicture}[thick]
\draw [white,line width=2.0pt,double=black,double distance=2.0pt] (2,4) to[out=180,in=270] (1,5);
\draw [white,line width=2.0pt,double=black,double distance=2.0pt] (2,5) to[out=315,in=90] (1,3);
\draw [white,line width=2.0pt,double=black,double distance=2.0pt] (2,5) to[out=135,in=90] (0,4);
\draw [white,line width=2.0pt,double=black,double distance=2.0pt] (3,4) to[out=90,in=180] (4,5);
\draw [white,line width=2.0pt,double=black,double distance=2.0pt] (5,3) to[out=90,in=0] (4,5);
\draw [white,line width=2.0pt,double=black,double distance=2.0pt] (2,2) to[out=0,in=180] (3,2);
\draw [white,line width=2.0pt,double=black,double distance=2.0pt] (4,3) to[out=270,in=0] (3,2);
\draw [white,line width=2.0pt,double=black,double distance=2.0pt] (3,4) to[out=270,in=90] (3,2);
\draw [white,line width=2.0pt,double=black,double distance=2.0pt] (2,3) to[out=0,in=270] (4,5);
\draw [white,line width=2.0pt,double=black,double distance=2.0pt] (4,3) to[out=90,in=0] (2,4);
\draw [white,line width=2.0pt,double=black,double distance=2.0pt] (4,5) to[out=90,in=90] (1,5);
\draw [white,line width=2.0pt,double=black,double distance=2.0pt] (2,0) to[out=0,in=270] (3,2);
\draw [white,line width=2.0pt,double=black,double distance=2.0pt] (2,0) to[out=180,in=270] (1,1);
\draw [white,line width=2.0pt,double=black,double distance=2.0pt] (3,1) to[out=0,in=270] (5,3);
\draw [white,line width=2.0pt,double=black,double distance=2.0pt] (3,1) to[out=180,in=270] (1,3);
\draw [white,line width=2.0pt,double=black,double distance=2.0pt] (2,3) to[out=180,in=270] (0,4);
\draw [white,line width=2.0pt,double=black,double distance=2.0pt] (2,2) to[out=180,in=90] (1,1);
\draw [black,-,line width = 2,decoration={markings,mark=at position 0.95 with {\arrow{>}}},postaction=decorate] (2,2) to[out=180,in=90] (1,1);
\filldraw [black] (2,2) circle [radius=0.1]; \node [right,anchor=south west] at (2,2) {B}; 
\filldraw [gray] (0.45,1.05) circle [radius=0.1]; \node [gray,right,anchor=north east] at (0.45,1.05) {$\infty$}; 
\draw [lightgray,line width=2,decoration={markings,mark=at position 0.23 with {\arrow{>}}},postaction=decorate] 
(2.0,1.8) -- (2.0,2.3) to[out=90,in=30] 
(1.3,2.5) -- (0.9,2.2) to[out=210,in=180]
(0.8,0.7) -- (1.3,0.7) to[out=0,in=240]
(2.3,0.9) -- (2.6,1.3) to[out=60,in=180]
(2.7,1.4) -- (3.2,1.4) to[out=0,in=-60]
(3.7,2.0) -- (3.4,2.4) to[out=120,in=0]
(3.2,2.7) -- (2.7,2.7) to[out=180,in=270]
(2.4,2.8) -- (2.4,3.3) --
(2.4,3.7) -- (2.4,4.2) to[out=90,in=180]
(2.8,4.6) -- (3.2,4.3) to[out=-30,in=60]
(3.3,4.1) -- (3.2,3.7) to[out=270,in=240]
(3.3,3.7) -- (3.5,4.1) to[out=60,in=-60]
(3.5,4.6) -- (3.2,5.0) to[out=120,in=30]
(2.2,5.2) -- (1.8,4.9) --
(1.4,4.6) -- (1.0,4.3) to[out=210,in=150]
(1.0,3.7) -- (1.4,3.4) to[out=-30,in=90]
(2.0,3.2) -- (2.0,2.7) to[out=-90,in=90]
(2.6,2.3) -- (2.6,1.8) to[out=-90,in=-90] (2.0,1.8);
\node [draw,black,fill=white,inner sep=1pt] at (1.05,0.7) {1}; 
\node [draw,black,fill=white,inner sep=1pt] at (2.5,1.1) {2}; 
\node [draw,black,fill=white,inner sep=1pt] at (3.0,1.45) {3}; 
\node [draw,black,fill=white,inner sep=1pt] at (3.6,2.2) {4}; 
\node [draw,black,fill=white,inner sep=1pt] at (3,2.65) {5}; 
\node [draw,black,fill=white,inner sep=1pt] at (2.4,3.1) {6}; 
\node [draw,black,fill=white,inner sep=1pt] at (2.4,4.0) {7}; 
\node [draw,black,fill=white,inner sep=1pt] at (3.1,4.5) {8}; 
\node [draw,black,fill=white,inner sep=1pt] at (3.35,4.0) {10}; 
\node [draw,black,fill=white,inner sep=1pt] at (3.2,3.8) {9}; 
\node [draw,black,fill=white,inner sep=1pt] at (3.4,4.9) {11}; 
\node [draw,black,fill=white,inner sep=1pt] at (2.0,5.0) {12}; 
\node [draw,black,fill=white,inner sep=1pt] at (1.2,4.4) {13}; 
\node [draw,black,fill=white,inner sep=1pt] at (1.3,3.6) {14}; 
\node [draw,black,fill=white,inner sep=1pt] at (1.9,3.0) {15}; 
\node [draw,black,fill=white,inner sep=1pt] at (2.6,1.95) {16}; 
\node [draw,black,fill=white,inner sep=1pt] at (1.1,2.3) {17}; 
\pgfresetboundingbox \clip (-0.5,-0.1) rectangle (5.5,6.3);
\end{tikzpicture} 
\\ 
\mc1l{(E) Performing suitable ~{\scriptsize\rotatebox[origin=c]{90}{$)\,($}} $\squigarrowleftright$ {\scriptsize$)\,($}~ moves}
&& 
\mc1l{(F) Labeling intersections by $C$ yields $\pi=$}
\\
\mc1l{leads to $C$, a separating curve.} 
&& 
\mc1l{$1,3,5,8,11,2,17,14,12,15,6,13,7,9,10,4,16$.}
\end{tabular}
}

\newcommand{\petaldiagrams}{
\begin{tabular}{cccc}
\tikz[thick,decoration={markings,mark=at position 0.05 with {\arrow{<}}}]{
\foreach \angle in {0, 120, ..., 240} \draw[postaction={decorate}] 
(\angle:0.5) .. controls +(\angle:1.0) and +(\angle+60:1.0) .. (\angle+60:0.5);
\foreach \angle in {0, 120, ..., 240} \draw (\angle:0.5) -- (\angle:-0.5);
\pgfresetboundingbox \clip (-1.5,-1.5) rectangle (1.5,1.5);}
&
\tikz[thick,decoration={markings,mark=at position 0.05 with {\arrow{<}}}]{
\foreach \angle in {0, 72, ..., 288} \draw[postaction={decorate}] 
(\angle:0.5) .. controls +(\angle:1.0) and +(\angle+36:1.0) .. (\angle+36:0.5);
\foreach \angle in {0, 72, ..., 288} \draw (\angle:0.5) -- (\angle:-0.5);
\pgfresetboundingbox \clip (-1.5,-1.5) rectangle (1.5,1.5);}
&
\tikz[thick,decoration={markings,mark=at position 0.05 with {\arrow{<}}}]{
\foreach \angle in {0, 51.5, ..., 309} \draw[postaction={decorate}] 
(\angle:0.5) .. controls +(\angle:1.0) and +(\angle+25.75:1.0) .. (\angle+25.75:0.5);
\foreach \angle in {0, 51.5, ..., 309} \draw (\angle:0.5) -- (\angle:-0.5);
\pgfresetboundingbox \clip (-1.5,-1.5) rectangle (1.5,1.5);}
&
\tikz[thick,decoration={markings,mark=at position 0.05 with {\arrow{<}}}]{
\foreach \angle in {0, 40, ..., 320} \draw[postaction={decorate}] 
(\angle:0.5) .. controls +(\angle:1.0) and +(\angle+20:1.0) .. (\angle+20:0.5);
\foreach \angle in {0, 40, ..., 320} \draw (\angle:0.5) -- (\angle:-0.5);
\pgfresetboundingbox \clip (-1.5,-1.5) rectangle (1.5,1.5);}
\end{tabular}}

\title{The Distribution of Knots in the Petaluma Model}

\author{
Chaim Even-Zohar 
\thanks{Department of Mathematics, University of California, Davis, California 95616. \href{mailto:chaim@math.ucdavis.edu}{chaim@math.ucdavis.edu}}
\and 
Joel Hass 
\thanks{Department of Mathematics, University of California, Davis, California 95616. \href{mailto:hass@math.ucdavis.edu}{hass@math.ucdavis.edu}}
\and 
Nati Linial 
\thanks{Department of Computer Science, Hebrew University, Jerusalem 91904, Israel. \href{mailto:nati@cs.huji.ac.il}{nati@cs.huji.ac.il}}
\and
Tahl Nowik
\thanks{Department of Mathematics, Bar-Ilan University, Ramat-Gan 5290002, Israel. \href{mailto:tahl@math.biu.ac.il}{tahl@math.biu.ac.il}}
\and
\thanks{This project was supported by BSF 2012188. The first author thanks ICERM and the Rothschild fellowship.}
}

\begin{document}

\maketitle

\begin{abstract}
The representation of knots by petal diagrams (Adams et al.~2012) naturally defines a sequence of distributions on the set of knots. In this article we establish some basic properties of this randomized knot model. We prove that in the random $n$-petal model the probability of obtaining every specific knot type decays to zero as $n$, the number of petals, grows. In addition we improve the bounds relating the crossing number and the petal number of a knot. This implies that the $n$-petal model represents at least exponentially many distinct knots. 

Past approaches to showing, in some random models, that individual knot types occur with vanishing probability, rely on the prevalence of localized connect summands as the complexity of the knot increases. However this phenomenon is not clear in other models, including petal diagrams, random grid diagrams, and uniform random polygons. Thus we provide a new approach to investigate this question.

\end{abstract}

\section{Introduction}

The study of random knots and links emerges from various perspectives, both theoretical and applied. See~\cite{models} for a survey of randomized knot models in the literature. We here pursue the study of the Petaluma model, based on petal diagrams~\cite{adams2015knot}. This model has the advantage of being based on one random permutation, and it seems related to knotting phenomena arising in biology and elsewhere. In a previous work~\cite{even2016invariants,even2016writhe} we investigated the distribution of finite type invariants in the Petaluma model. Here we return to some remaining fundamental questions about this model, such as how many knots can appear and with what probabilities. 

\begin{figure}[H]
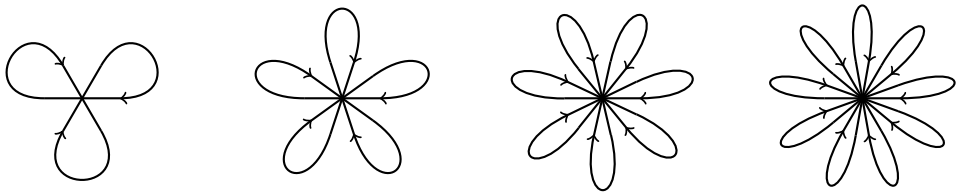

\begin{center}
\petaldiagrams
\caption{Petal diagrams with 3, 5, 7 and 9 petals.}
\label{petals}
\end{center}
\end{figure}

Consider a \emph{petal diagram} with an odd number $2n+1$ of petals, as in Figure~\ref{petals}. Each assignment of $2n+1$ heights to the $2n+1$ straight arcs above the multi-crossing point determines a knot. Indeed, the knot type is well-defined by the relative ordering of the heights, through a smooth curve in $\R^3$ that projects to this diagram. The random variable $K_{2n+1}$ is the knot type obtained from a uniformly random sequence of heights $\pi \in S_{2n+1}$. These heights $\pi(1), \pi(2), \dots$ correspond to the straight arcs in the order they occur as one travels along the diagram. 

For example, $K_3$ is the \emph{unknot} with probability $1$, while $K_5$ yields a \emph{trefoil} knot with probability $1/12$. This is obtained by the permutation $\pi=(1,3,5,2,4) \in S_5$ and by its rotations and reflections. The permutation $(1,5,3,7,2,4,6) \in S_7$ yields the \emph{figure eight} knot. Adams et al.~give more examples, and prove that all knots appear in this model.

\begin{thm}[\cite{adams2015knot}]\label{existential}
For every knot $K$ there exists an odd $p \in \N$ such that $K$ can be realized by a $p$-petal diagram with a permutation $\pi \in S_p$. 
\end{thm}

In fact, the knot $K$ will then have a $q$-petal diagram for every odd $q \geq p$. This follows easily by inserting two consecutive heights to the permutation, without changing the knot type. The smallest such $p$ is denoted $p(K)$, the \emph{petal number} of $K$. 

The efficiency of this representation is studied by relating it to regular \emph{knot diagrams}, which are planar projections that are one-to-one except for finitely many transverse double points. The \emph{crossing number}, denoted $c(K)$, is the least number of such \emph{crossings} in any diagram of $K$. 

\begin{example*} 
Adams et al.~\cite{adams2015knot} precisely compute the petal number for two infinite families of knots.
\begin{enumerate}
\item 
$T_n = $ the $(n,n+1)$-torus knot,~ $p(T_n) = 2n+1$,~ $c(T_n) = n^2-1$.
\item
$S_n = $ the $(2n-1)$-twist knot,~ $p(S_n) = 2n+1$, $c(S_n) = 2n-1$.
\end{enumerate}
\end{example*}

These explicit constructions are optimal for petal representations in the following sense.

\begin{thm}[\cite{adams2015knot,adams2015bounds}] \label{crossing}
For every non-trivial knot $K$,
$$ c(K) \;\leq\; \left\lfloor \frac{p(K)}{2}\right\rfloor^2 - 1 \;.$$
If $K$ is alternating, then
$$ c(K) \;\leq\; p(K)-2 \;.$$
\end{thm}
 
Here we consider the opposite direction and give a quantitative bound on the worst case, which turns out to be only a constant factor away from the second example above.

\begin{thm}\label{petal}
For every non-trivial knot $K$,
$$ p(K) \;\leq\; 2\,c(K)-1 \;.$$
\end{thm}

Theorem~\ref{petal} is proved in Section~\ref{cross} via the analysis of an efficient algorithm that transforms a regular knot diagram into a petal diagram with a suitable permutation. This, in particular, yields a constructive proof of Theorem~\ref{existential}. 

Theorem~\ref{petal} shows that $(2n-1)$-petal diagrams represent at least as many knots as regular knot diagrams with $n$ crossings. Some explicit constructions are known to generate $\Omega(2.68^n)$ different $n$-crossing knots~\cite{welsh1991number}. Consequently,

\begin{cor}\label{many}
There are at least $\Omega(2.68^n)$ distinct $(2n-1)$-petal knots.
\end{cor}

\medskip

A natural question that applies to any random model of knots, asks for the probability of generating the unknot. This goes back to the oldest models of random knots, by Delbruck~\cite{delbruck1961knotting} and by Frisch and Wasserman~\cite{frisch1961chemical}, that are based on certain types of polygonal paths in~$\Z^3$ and in~$\R^3$. The Delbruck--Frisch--Wasserman Conjecture asserts that the resulting knot is non-trivial \emph{with high probability}, i.e., the probability of the unknot decays to zero as the number of steps grows. This was positively settled in various models by finding small localized connected summands in the prime decomposition of the knot~\cite{sumners1988knots,pippenger1989knots,soteros1992entanglement,diao1994random,diao1995knotting}. Similar reasoning worked for another model, based on random planar diagrams~\cite{chapman2016asymptotic2}. However, we don't expect this behavior of the prime decomposition in the Petaluma model, and hence we have to use another knot invariant.

The \emph{Casson invariant} $\mathrm{c}_2(K)$ is the coefficient of $x^2$ in the Alexander-Conway polynomial $C_K(x)=1+\mathrm{c}_2x^2+\dots$~\cite{lickorish1997introduction}. It is also the unique second order invariant of knots, up to affine transformations~\cite{chmutov2012introduction}. In the Petaluma model we showed $E[\mathrm{c}_2(K_{2n+1})^k] = \mu_k n^{2k} + O(n^{2k-1})$, and obtained formulas for the normalized limiting moments $\mu_k$~\cite{even2016invariants}. However, it is impossible to conclude that with high probability $\mathrm{c}_2$ does not vanish based solely on finitely many limiting moments~\cite{kreuin1977markov}. Here we overcome this difficulty.

\begin{thm}\label{c2}
For every $n \in \N$ and $v \in \Z$,
$$ P\left[\mathrm{c}_2\left(K_{2n+1}\right) = v\right] \;\leq\; \frac{8}{n^{1/10}} \;.$$
Consequently, for every knot $K$, 
$$ P\left[K_{2n+1} = K\right] \; \xrightarrow{\; n\rightarrow\infty \;} \; 0\;,$$
and in particular, $K_{2n+1}$ is knotted with high probability.
\end{thm}

Theorem~\ref{c2} is proved in Section~\ref{decay}. Our approach involves the analysis of the formulas for the Casson invariant and for the linking number, evaluated on random knots and links, together with a simple coupling argument. As these invariants are given by summation over all crossings, we show that a small perturbation of the heights' ordering is likely to spread their distribution over many values, and deduce that they cannot be too concentrated. 

In our proof, we establish a similar bound,  $P\left[\mathrm{lk}\left(L_{2m,2n}\right) = v\right] \leq 6/\sqrt{\min(m,n)}$, for the linking number of a random $(2m,2n)$-petal link. See Section~\ref{decay} for precise definitions and statement.

\medskip
Our approach to the Delbruck--Frisch--Wasserman conjecture is different than those previously applied to other constructions of random knots, as it doesn't rely on establishing the presence of small connected summands. Indeed, the occurrence of such summands seems less likely in the Petaluma model, where the typical ``step-length'' is comparable to the diameter of the whole curve. We thus expect our methods to extend to other well-studied knot models, in which local entanglements are similarly believed to be rare. 

For example, two random permutations $\pi,\sigma \in S_n$ define a knot via an $n \times n$ \emph{grid diagram}~\cite{brunn1897uber,cromwell1998arc}. See also~\cite{even2016invariants}. It is very plausible that an adaptation of our argument, based on perturbing one of these permutations, would yield a proof of the still-open Delbruck--Frisch--Wasserman conjecture in this setting. Another case in point is Millett's \emph{uniform random polygon}~\cite{millett2000monte} with $n$ segments, where the spatial confinement to the cube seems to decrease local knotting, and the conjecture is yet to be verified.

\medskip
Further discussion and open questions appear at the end of each section.


\section{Petal Number and Crossing Number}\label{cross}

We first give a simple proof of $p(K) < 4c(K)$, and then improve it to $p(K) < 2c(K)$ with a more technical argument.

\begin{proof}[Proof of $p(K)<4c(K)$]
The proof refines the construction of petal diagrams by Adams et al. ~\cite{adams2015knot}. Basically, we preprocess the planar embedding of a minimal-crossing knot diagram, and then run their algorithm.

Consider a knot diagram with $n$ crossings. Travel along the knot diagram starting from some base point~$B$. Mark each crossing as ascending ($A$) or descending ($D$), depending on whether its lower or upper strand is visited first. See Figure~\ref{diagrams}A. Note that if all the vertices have the same type then $K$ is the unknot.

This yields two finite sets of points in the plane which can be separated by a generic simple closed curve $C$. We assume that $C$ passes through $B$, avoids all the crossing points of the diagram, and crosses it transversely finitely many times. It simplifies the construction to assume that the point at infinity lies on $C$, either by choosing it accordingly or by applying an isotopy of the diagram in $S^2$.  

Let $E$ be any \emph{edge} of the knot diagram, viewed as a $4$-regular plane graph. We claim that the separating curve~$C$ can be chosen so that it does not intersect~$E$ more than twice. Indeed, if they intersect three times, then the local operation shown in Figure~\ref{switch} keeps~$A$ and~$D$ separated and~$C$ connected. Note that this is the only possible configuration of three adjacent intersection points on~$E$, up to rotations and reflections. If one of them is the base point~$B$ then we can let the new curve~$C$ pass through it again. Repeating for all $2n$ edges as needed, $C$ intersects the knot at $\leq 4n$ points.

\begin{figure}[b]
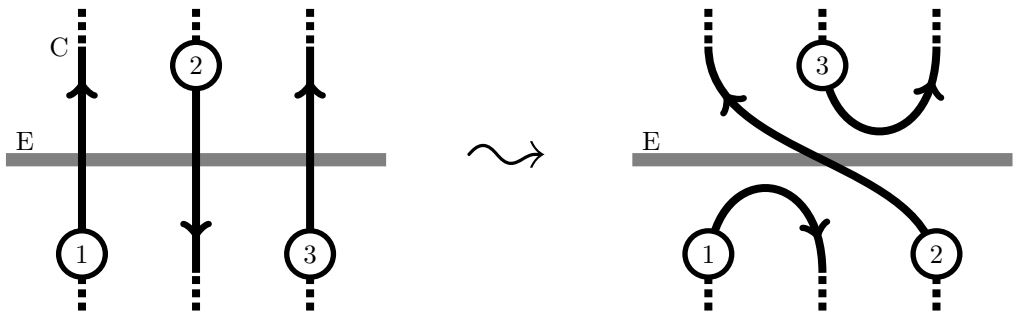

\begin{center}
\switch
\caption{Reducing intersections. The \emph{cyclic} ordering of parts of~$C$ is marked by $1,2,3$.}
\label{switch}
\end{center}
\end{figure}

The rest of the construction is almost unchanged from~\cite{adams2015knot}, and we briefly repeat its main steps. The reader is referred to~\cite{adams2015knot} for more details and illustrations.
\begin{enumerate}
\itemsep0em 
\item 
Isotope the plane diagram so that $C$ is the $y$-axis, ascending crossings have positive $x$-coordinate, and descending crossings have negative $x$-coordinate.
\item 
Further isotope the knot diagram so that for $|x| \leq 1$ it consists of an even number of horizontal segments, intersecting the $y$-axis.
\item 
Start at $B=(0,y_0)$ and travel along the knot diagram. The point $B$ and the other $p$ intersection points $(0,y_1),\dots,(0,y_p)$ cut the knot into $p+1$ arcs. Denote $y_{p+1}=y_0$.
\item 
Lift the diagram to $\R^3$ so that above $(0,y_k)$ it remains a straight segment $z=kx$ for $|x|\leq 1$, and the angle $\arctan (z/x)$ is non-decreasing within each arc. 
\item 
Instead of lifting $B$ to $(0,y_0,0)$, with two half segments to~$(1,y_0,0)$ and~$(-1,y_{p+1},-(p+1))$, connect these two points directly by a straight line segment. 
\end{enumerate}
Note that step $4$ preserves the knot type thanks to the condition on the $A$ and~$D$ crossings. Indeed, by the choice of line segments between arcs and lifting within them, $z/x$ is increasing throughout our travel along the knot. Since $x>0$ in each ascending point, the $z$-coordinate of the lifted curve rises between its two visits to such point. The case of decreasing points is similar. 

It follows from the construction that the projection of the lifted knot along the $y$-direction to the $xz$-plane yields a single multi-crossing point at the origin. Moreover, the projected curve has $(p-1)/2$ petals contained in the first quadrant, $(p-1)/2$ petals in the third quadrant, and one fat petal encompassing the fourth quadrant. This is a petal diagram with $p<4n$ petals. 
\end{proof}

To summarize, the construction in the above proof underlies the following \emph{petal algorithm}, cf.~\cite{adams2015bounds}, representing a knot by a petal diagram given a regular diagram. 
\begin{itemize}
\itemsep0em
\item[$\blacktriangleright$]
Travel along the diagram starting at some point $B$, and identify crossing types in $\{A,D\}$. 
\item[$\blacktriangleright$]
Choose a generic curve~$C$, containing~$B$, separating~$A$'s from~$D$'s.
\item[$\blacktriangleright$]
Travel along~$C$ and label with~$1,\dots,p$ all intersections with the diagram other than~$B$. 
\item[$\blacktriangleright$]
Travel again along the diagram and record the ordering of the labels as $\pi \in S_p$.
\end{itemize}
\begin{rmk*}
By the above argument, in the third and fourth steps one has to travel along the curve~$C$ starting from the point at infinity, and along the knot diagram starting from the base point~$B$. However, as observed in~\cite{adams2015bounds}, the resulting knot type is preserved under rotations of $\pi$ from both directions: $K_p(\pi) = K_p(\pi \circ \rho) = K_p(\rho \circ \pi)$ where $\rho(x) = (x+1)\bmod p$. It follows that these two starting points can be chosen arbitrarily.
\end{rmk*}

We have shown that there always exists such a~$C$ with at most~$4n$ intersections. Now we improve the bound by choosing~$C$ even more efficiently.

\begin{proof}[Proof of $p(K)<2c(K)$]
The upper bound is derived by better controlling the number $\in \{0,1,2\}$ of intersection points on each edge. Since the curve $C$ is separating, an edge between two vertices of type~$A$ and $D$ contributes one point of intersection. We construct $C$ more carefully, such that the edges that are disjoint from it can be at least as frequent as those with two intersection points.

\begin{figure}
\begin{center}
\diagrams
\caption{Different stages of the construction in the proof of Theorem~\ref{petal}. Here $K = 3_1\#4_1$.}
\label{diagrams}
\end{center}
\end{figure}

Consider a knot diagram of $K$ with $n$ crossings, not necessarily one that realizes the crossing number. As before, we view $K$ as a $4$-regular plane graph with a base point $B$ on one edge, and vertices of type~$A$ or~$D$, as in Figure~\ref{diagrams}A.

It is sufficient to prove $p(K) < 2n$ for the case of \emph{$3$-edge-connected} diagrams, ones that remain connected whenever $2$ edges are removed. Otherwise, find a \emph{$2$-edge cut} that disconnects the diagram. It follows that $K$ is a connected sum $K_1\#K_2$ with diagrams of $n_1$ and $n_2$ crossings respectively, where $n_1+n_2=n$. By the sub-additivity of petal numbers~\cite[Theorem 2.4]{adams2015knot} and induction on $n$, we have $ p(K) < p(K_1) + p(K_2) < 2n_1 + 2n_2 = 2n$. We note that if $K$ is prime and given by a diagram with $c(K)$ crossings then it is already $3$-edge-connected.

\medskip
As usual, the vertices of the \emph{dual graph} $K'$ correspond to the faces of the diagram, and edges correspond to edges. $K'$ is \emph{simple}, without loops and multiple edges, since the knot diagram is $3$-edge connected. It is also \emph{bipartite} since the diagram is $4$-regular and planar, so that its faces admit a checkerboard coloring.

We define a subgraph $G$ of $K'$, with the same set of vertices, whose edges are those that correspond to the $A$-$D$ edges in the diagram. Denote by $2m$ the number of edges as this number must be even. Recall that we are assuming $K$ is nontrivial, and so $m>0$. Note also that there is an even number of $A$-$D$ edges around each face of the knot diagram, hence the vertices of $G$ have even degrees. $G$ is planar, simple and bipartite since $K'$ is. See Figure~\ref{diagrams}B for an example of such~$G$.

Let $H$ be a connected component of $G$ with $k>0$ edges, which implies $k > 1$ by the above. It is a corollary of Euler's formula that the number of vertices and edges in a connected simple bipartite plane graph satisfy the relation $2v \geq e+4$, since each face has at least four sides. Hence $H$ has at least $k/2+2$ vertices, and a spanning tree with at least $k/2+1$ edges.

The union of spanning trees for all such $H$'s yields a forest $F$ in $G$ with at least $m+1$ edges. By adding at most $n-m$ edges of type $A$-$A$ or $D$-$D$, we complete $F$ to a spanning tree $T$ of the whole $(n+2)$-vertex graph $K'$. 

We now describe how to make sure that either $G$ or $T$ contains the edge that corresponds to the base point. If it is an $A$-$D$ edge then we are fine as $G$ contains it. If its two adjacent faces are in different connected components of $G$, then we can pick this edge when choosing $T$. Otherwise, we throw this edge into $G$ even though its type is $A$-$A$ or $D$-$D$, so that $G$ has $2m+1$ edges. Repeating the above computation with $2m+1$ in place of $2m$ shows that in fact $F$ has $\geq m+2$ edges. Therefore at most $n-m-1$ further $A$-$A$ or $D$-$D$ edges were needed to construct $T$. See Figure~\ref{diagrams}B again for an example of this latter scenario.
 
In conclusion, $G \cup T$ is connected and spans the dual graph $K'$, with $2m$ edges of type $A$-$D$ and at most $n-m$ edges of type $A$-$A$ or $D$-$D$.

\medskip
We construct $C$ so that it intersects the knot diagram exactly in the edges corresponding to~$G \cup T$. We start by putting one small line segment across every $A$-$D$ edge, and two small line segments across every $A$-$A$ or $D$-$D$ edge in $G \cup T$, as in Figure~\ref{diagrams}C. The total number of line segments is at most $ 2m \cdot 1 + (n-m) \cdot 2 = 2n$. 

Since an even number of line segments emanate into each face, we can match their tips to each other without crossings, say by connecting adjacent ones.  We have thus separated the $A$'s from the $D$'s by a set of disjoint embedded circles. See Figure~\ref{diagrams}D.

Different circles may be cut and reconnected together via the local moves ~{\scriptsize\rotatebox[origin=c]{90}{$)\,($}} $\squigarrowleftright$ {\scriptsize$)\,($}~ as long as they pass through a common face. Since the graph $G \cup T$ connects all faces of the diagram, we can perform such moves until we end up with one long circle $C$, as in Figure~\ref{diagrams}E.

Since the curve $C$ intersects each $A$-$D$ edge once and each $A$-$A$ or $D$-$D$ edge twice, it separates the ascending and descending crossings. By construction, $C$ intersects the knot at~$\leq 2n$ points, and visits both the base point and the point at infinity if specified. We apply to it the algorithm by Adams et al.~as above. See Figure~\ref{diagrams}F.
\end{proof}

\begin{disc*}
Several questions remain open.
\begin{enumerate}
\item 
How tight is Theorem~\ref{petal}? It is tight for $c(K)=3$ or~$4$, but this is presently known in general only up to a factor of two, by Theorem~\ref{crossing}.  
\item 
As observed in Corollary~\ref{many} there are at least exponentially many distinct $n$-petal knots. Is this asymptotic estimate tight? We cannot, at present, rule out the possibility that the answer is, in fact, $\exp\left(\Omega(n \log n)\right)$.
\item 
What is the typical crossing number of $K_{2n+1}$ in the Petaluma model? By Theorems~\ref{crossing}-\ref{petal}, the crossing number of a knot is between linear and quadratic in its petal number. 

Experiments by the authors and by Adams and Kehne~\cite{adams2016bipyramid,uberluma} indicate that the hyperbolic volume of $K_{2n+1}$ is typically of order $n \log n$, which yields a similar lower bound for the typical $c(K_{2n+1})$. See~\cite{dunfield2014random} for experiments on the hyperbolic volume and the number of crossings in a different model.
\end{enumerate}
\end{disc*}

\section{Petaluma Knots are Knotted}\label{decay}

The proof of Theorem~\ref{c2} relies on an analogous and easier statement concerning the linking number of $2$-component links in the Petaluma model. 

A random 2-component link $L_{2m,2n}$ is obtained from a petal diagram as in Figure~\ref{fig-moves}B, with $2m$ and $2n$ petals in the black and grey components respectively. A uniformly random permutation $\pi \in S_{2m+2n}$ determines the height of the arcs above the center point.

Recall that the \emph{linking number} of a 2-component link $L$ can be defined in terms of a link diagram of $L$ as the sum of crossing signs: $\mathrm{lk}(L) = \tfrac12\left(\#\tikz[thick,->,baseline=-2]{\draw[black](0.3,-0.1) -- (0,0.2);\draw[white,-,line width=3](0.03,-0.1) -- (0.33,0.2);\draw[black](0.03,-0.1) -- (0.33,0.2);}-\#\tikz[thick,->,baseline=-2]{\draw[black](0.03,-0.1) -- (0.33,0.2);\draw[white,-,line width=3](0.3,-0.1) -- (0,0.2);\draw[black](0.3,-0.1) -- (0,0.2);}\right)$ where only crossings between the two components are counted. In our previous work we found the limiting distribution of the linking number in the Petaluma model~\cite{even2016invariants}.

The analogue of Theorem~\ref{c2} for 2-component links is the following bound.

\begin{thm}\label{lk}
For every $m,n \in \N$ and $v \in \Z$,
$$ P\left[\mathrm{lk}\left(L_{2m,2n}\right) = v\right] \;\leq\; \frac{6}{\sqrt{\min(m,n)}} \;.$$
\end{thm}

We shall use more than once in our proofs the following classical result by Erd\H{o}s, known as the Littlewood--Offord problem.

\begin{thm}[\cite{erdos1945lemma}] \label{lo}
Let $a_1,\dots,a_t \in \R$. At most $\tbinom{t}{\lfloor t/2 \rfloor}$ of the $2^t$ sums $\left\{\sum_{i \in I} a_i : I \subseteq \{1,\dots,t\}\right\}$ are contained in any open interval of length $\min_i|a_i|$.
\end{thm}

In the language of probability, if one of the $2^t$ sums is sampled uniformly at random, then it is contained in such an interval with probability at most
$\tbinom{t}{\lfloor t/2 \rfloor}/2^t \leq 1/\sqrt{t} $.

\begin{proof}[Proof of Theorem~\ref{lk}]
Consider a random 2-component link $L_{2m,2n}$. Denote the $2m$ and $2n$ straight arcs at the center of the petal diagram by $I = \{1,\dots,2m\}$ and $J = \{2m+1,\dots,2m+2n\}$ respectively. As usual, let $\pi \in S_{2m+2n}$ be the random heights of these arcs, and denote for brevity $\mathrm{lk}(\pi) = \mathrm{lk}\left(L_{2m,2n}(\pi)\right)$. Perturb the petal diagram near the center point to obtain a regular link diagram. 

\begin{samepage}
By the crossing signs formula for the linking number,
$$ \mathrm{lk}(\pi) \;=\; \tfrac12\left(\#\;\tikz[scale=1.25,thick,->,baseline=-1]{\draw[black](0.3,-0.1) -- (0,0.2);\draw[white,-,line width=4](0.03,-0.1) -- (0.33,0.2);\draw[black](0.03,-0.1) -- (0.33,0.2);}\;-\;\#\;\tikz[scale=1.25,thick,->,baseline=-1]{\draw[black](0.03,-0.1) -- (0.33,0.2);\draw[white,-,line width=4](0.3,-0.1) -- (0,0.2);\draw[black](0.3,-0.1) -- (0,0.2);}\;\right) \;=\; \tfrac12 \;\sum\limits_{i \in I}\;\sum\limits_{j \in J} (-1)^{i+j} \cdot \begin{cases} +1 & \pi(i) > \pi(j) \\ -1 & \pi(i) < \pi(j) \end{cases} $$
where the $(i,j)$ term corresponds to the crossing of arc $i$ from the first component and arc $j$ from the second component. The sign of such a crossing depends on the heights $\pi(i),\pi(j)$ of these two arcs, and also on their orientations as determined by the parity of $i$ and $j$, see Figure~\ref{fig-moves}B.
\end{samepage}

The proof of Theorem~\ref{lk} goes by perturbing the permutation $\pi$ with $m+n$ \emph{swaps}, which are transpositions of arcs with adjacent heights, such as~$\pi' = (1\;2) \circ \pi$. By the above formula, the effect of such swaps is $\pm 1$ for a \emph{mixed} pair of arcs, with one arc from each components, and $0$ otherwise. The contributions of disjoint swaps are additive. 

We proceed by the following procedure. We first pick $\pi$ uniformly at random from all $(2m+2n)!$ permutations. Then we obtain $\pi'$ from $\pi$ by swapping, via a random subset of $\{(1\;2), (3\;4), \dots\}$, uniformly chosen from all $2^{m+n}$ subsets. Note that if $\pi$ is uniformly distributed then so is~$\tau \circ \pi$ for any fixed~$\tau$. Therefore, the distribution of~$\pi'$ is a mixture of uniform distributions, which is uniform as well.

Starting from $\mathrm{lk}(\pi)$, each mixed pair contained in this random subset changes the linking number by $\pm 1$. Therefore, the probability that $\mathrm{lk}(\pi')$ attains a given value $v$ is bounded by Theorem~\ref{lo}, in the easy special case where all $a_i = \pm 1$. If there are $t$ mixed pairs in $\pi$ then
$ P\left[\, \mathrm{lk}(\pi')=v \;|\; t \,\right] \;\leq\; 1/\sqrt{t} $.

It is hence useful to derive a lower bound on $t$, the number of mixed pairs of arcs out of all $m+n$ pairs under consideration. Lemma~\ref{match} below claims that the probability of having less than $\min(m,n)/2$ mixed pairs is at most $20/\min(m,n)$.

Finally, we divide into two cases, according to whether or not $t \geq \min(m,n)/2$. Applying the union bound,
$$ P\left[ \mathrm{lk}(\pi')=v\right] \;\leq\; \frac{1}{\sqrt{\min(m,n)/2}} + \frac{20}{\min(m,n)} \;\leq\; \frac{2+4}{\sqrt{\min(m,n)}} $$ 
where we use the observation that the proposition is trivially true for $\sqrt{\min(m,n)} \leq 5$.
\end{proof}

\begin{lemma}\label{match}
Let $I$ and $J$ be two disjoint non-empty sets of cardinality $2m$ and $2n$ respectively. Let $E = (e_1,\dots,e_{m+n})$ be a random matching of $I \cup J$. Then the probability that less than $\min(m,n)/2$ edges in the matching connect elements of $I$ and $J$ is at most $20/\min(m,n)$.
\end{lemma}

\begin{proof}
Denote by $Z$ the count of edges that \emph{mix} $I$ and $J$, meaning that they connect an element of $I$ with an element of $J$. Chebyshev's inequality for $Z$ will be sufficient for our argument~\cite{ross2009first}. We estimate the expectation and variance of $Z$. Denote $Z = Z_1 + \dots + Z_{m+n}$ where $Z_i = 1$ if the edge $e_i$ is mixed and $0$ otherwise.  
$$ E[Z_i] \;=\; \frac{2m \cdot 2n}{\tbinom{2m+2n}{2}} \;\geq\; \frac{2mn}{(m+n)^2} \;\;\;\;\;\;\Rightarrow\;\;\;\;\;\; E[Z] \;\geq\; \frac{2mn}{m+n} $$
Note that we may assume $m,n \geq 20$, as otherwise the lemma clearly holds.
$$ V[Z_i] \;=\; \frac{4mn}{\tbinom{2m+2n}{2}} - \left(\frac{4mn}{\tbinom{2m+2n}{2}}\right)^2 \;\leq\; \frac{4mn}{2(m+n)^2\left(1-\tfrac{1}{2(m+n)}\right)} \;\leq\; \frac{2mn}{(m+n)^2 \cdot \left(\tfrac{79}{80}\right)} \;\leq\; \frac{3\,mn}{(m+n)^2} $$
In the covariance of $Z_i$ and $Z_j$ for $i \neq j$ the terms of order $m^2n^2/(m+n)^2$ cancel, and one can show by similar estimates 
$$ COV[Z_i,Z_j] \;=\; \frac{2\,\tbinom{2m}{2}\tbinom{2n}{2}}{3\,\tbinom{2m+2n}{4}} - \frac{(4mn)^2}{\tbinom{2m+2n}{2}^2} \;\leq\; \frac{14\,m^2n^2}{(m+n)^5} $$
Therefore
$$ V[Z] \;=\; \sum\limits_{i=1}^{m+n} V[Z_i]  + \sum\limits_{i \neq j} COV[Z_i,Z_j] \;\leq\; \frac{3\,mn}{m+n} + \frac{14\,m^2n^2}{(m+n)^3} \;\leq\; \frac{10\,mn}{m+n} $$
By Chebyshev's inequality,
$$ P\left[Z \leq \frac{\min(m,n)}{2}\right] \;\leq\; P\left[Z \leq \frac{E[Z]}{2}\right] \;\leq\; \frac{4\,V[Z]}{E[Z]^2} \;\leq\; \frac{10(m+n)}{mn} \;\leq\; \frac{20}{\min(m,n)} $$
as required.
\end{proof}

The proof of Theorem~\ref{c2} goes by random arc swaps as well, and makes use of the notion of smoothing. By properties of the Alexander--Conway polynomial, the effect of a crossing change on the Casson invariant is given by the linking number of the \emph{smoothing} of that crossing, which is the 2-component link obtained from reconnecting the two strands. This can be summarized by the relation $\mathrm{c}_2(\tikz[thick,->,baseline=-2]{\draw[black](0.3,-0.1) -- (0,0.2);\draw[white,-,line width=3](0.03,-0.1) -- (0.33,0.2);\draw[black](0.03,-0.1) -- (0.33,0.2);}) - \mathrm{c}_2(\tikz[thick,->,baseline=-2]{\draw[black](0.03,-0.1) -- (0.33,0.2);\draw[white,-,line width=3](0.3,-0.1) -- (0,0.2);\draw[black](0.3,-0.1) -- (0,0.2);}) = \mathrm{lk}(\mathlarger{\mathlarger{\rcurvearrowup\lcurvearrowup}})$ where the rest of the diagram is the same.  

Let $\pi,\tau \in S_{2n+1}$ where $\tau = (t\;\;t+1)$ is an adjacent transposition. Consider the smoothed link $L(\pi,\tau)$ obtained from the knot $K_{2n+1}(\pi)$ by reconnecting the two arcs at heights $t$ and $t+1$. The following lemma shows that $L(\pi,\tau)$ has a petal representation which is closely related to that of the given knot.

For such $\pi$ and $\tau$ we denote $d = \left\lfloor |\pi^{-1}(t+1) - \pi^{-1}(t)|/2\right\rfloor$. In other words, $d$ is half the distance between the locations of $t$ and $t+1$ in $\pi$.

\begin{lemma} \label{smooth}
For any $\pi \in S_{2n+1}$ and a transposition $\tau = (t\;\;t+1) \in S_{2n+1}$, the smoothed link $L(\pi,\tau)$ is given by $L_{2m,2(n-m)}(\pi_t)$ for some $\pi_t \in S_{2n}$ where $m = d$ or $n-d$. 

Moreover, if $\pi$ is uniformly random, then the conditional distribution of the smoothed link given the value of $m$ is the same as $L_{2m,2(n-m)}(\sigma)$ for uniform $\sigma \in S_{2n}$.
\end{lemma}

\begin{proof}
Observe that knots and links given as petal diagrams are invariant under \emph{vertical rotation}, e.g.~$K_{2n+1}(\pi) = K_{2n+1}(\rho \circ \pi)$ where $\rho(i) = (i+1)\bmod (2n+1)$. Thus we may assume without loss of generality that $t = 2n$, so that the smoothing takes place between the two highest arcs in $\R^3$. 

Smoothing the top two arcs might introduce new crossing points to the diagram, in addition to the single multi-crossing point. However, since the smoothed arcs are above the rest, they can be taken outside, to the top part of the diagram, as demonstrated in Figure~\ref{fig-smooth}. This operation creates a petal diagram of the two-component link $L(\pi,\tau)$, with two large outer petals.

In order to study the permutation of the resulting petal link diagram, we equivalently describe the smoothing using the two steps shown in Figure~\ref{fig-moves}. First, we take outside the highest arc as in Figure~\ref{fig-moves}A. Now, the arc that is at height $2n$ can be any of the remaining $2n$ arcs above the center. Consider the unique petal continuing that arc in the top part of the petal projection. The smoothing by $\tau = (2n\;\;2n+1)$ is now performed by reconnecting this petal with the large outer petal. These are the dotted petals in Figure~\ref{fig-moves}B.

\begin{figure}[t]
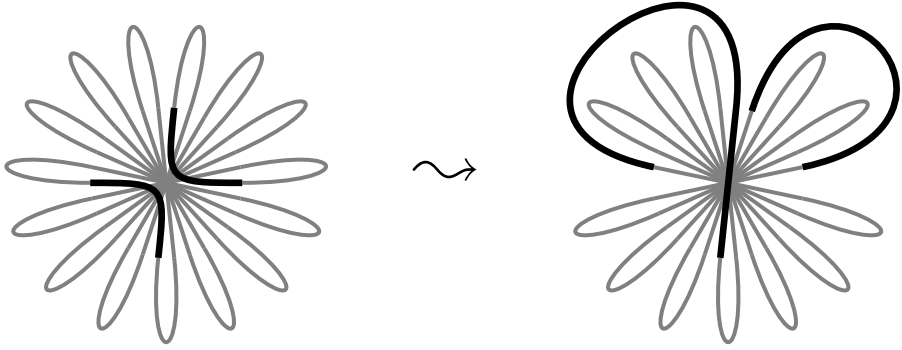

\begin{center}
\tikz[scale=1]{
\foreach \angle in {0, 24, 48, 72, 96, 120, 144, 168, 192, 216, 240, 264, 288, 312, 336} \draw[line width=1.5, color=gray] (\angle:1) .. controls +(\angle:1.5) and +(\angle+12:1.5) .. (\angle+12:1);
\foreach \angle in {24, 48, 72, 96, 120, 144, 168, 192, 216, 240, 288, 312, 336} \draw[line width=1.5, color=gray] (\angle:1) -- (\angle:-1);
\draw[line width=2.5, color=black] (0:1) .. controls +(0:-1) and +(84:-1) .. (84:1);
\draw[line width=2.5, color=black] (180:1) .. controls +(180:-1) and +(264:-1) .. (264:1);
\pgfresetboundingbox \clip (-2.25,-2.5) rectangle (2.25,2.5);} 
\raisebox{70pt}{\Huge \;\;\;\;$\leadsto$\;\;\;\;} 
\tikz[scale=1]{
\foreach \angle in {24, 48, 96, 120, 144, 192, 216, 240, 264, 288, 312, 336} \draw[line width=1.5,
color=gray] (\angle:1) .. controls +(\angle:1.5) and +(\angle+12:1.5) .. (\angle+12:1);
\foreach \angle in {24, 48, 72, 96, 120, 144, 168, 192, 216, 240, 288, 312, 336} \draw[line width=1.5, color=gray] (\angle:1) -- (\angle:-1);
\draw[line width=2.5, color=black] (12:1) .. controls +(12:2.8) and +(72:2.8) .. (72:1);
\draw[line width=2.5, color=black] (168:1) .. controls +(168:3.2) and +(84:3.2) .. (84:1) -- (264:1);
\pgfresetboundingbox \clip (-2.25,-2.5) rectangle (2.25,2.5);} 
\end{center}
\captionsetup {width=0.9\textwidth}
\caption{Smoothing and taking outside the highest two strands in a petal diagram.}
\label{fig-smooth}
\end{figure}

\begin{figure}[bt]
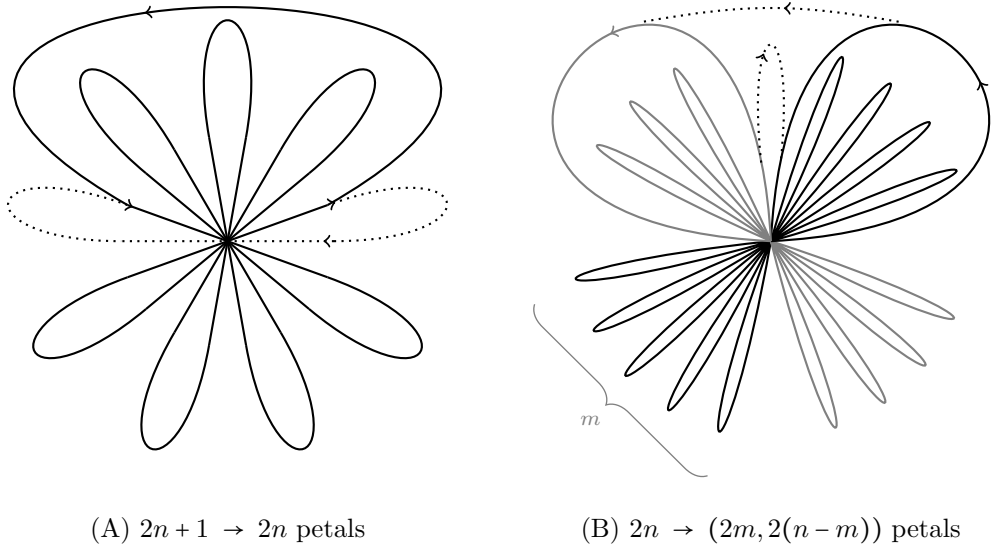

\begin{center}
\begin{tabular}{ccc}
\tikz[thick, scale=1.4]{
\foreach \angle in {40, 80, 120, 200, 240, 280, 320} \draw (\angle:1) .. controls +(\angle:1.5) and +(\angle+20:1.5) .. (\angle+20:1);
\draw[postaction={decorate},dotted,decoration={markings,mark=at position 0.995 with {\arrow{<}}}] (0:1) .. controls +(0:1.5) and +(20:1.5) .. (20:1);
\draw[postaction={decorate},dotted,decoration={markings,mark=at position 0.005 with {\arrow{<}}}] (160:1) .. controls +(160:1.5) and +(180:1.5) .. (180:1);
\foreach \angle in {40, 80, ..., 320} \draw (\angle:1) -- (\angle:-1);
\draw[dotted,postaction={decorate},dotted,decoration={markings,mark=at position 0.05 with {\arrow{>}}}] (0:1) -- (0:-1);
\draw[postaction={decorate},decoration={markings,mark=at position 0.6 with {\arrow{>}}}] (20:1) .. controls +(30:5) and +(150:5) .. (160:1);
\pgfresetboundingbox \clip (-2.25,-2.5) rectangle (2.25,2.5);} & &
\tikz[thick, decoration={markings,mark=at position 0.45 with {\arrow{>}}}, scale=1.4]{
\foreach \angle/\color in {18/black,34/black,50/black,66/black,116/gray,132/gray,148/gray,
188/black,204/black,220/black,236/black,252/black,286/gray,302/gray,318/gray,334/gray} 
\draw[color=\color] (0,0) .. controls +(\angle-2:2.5) and +(\angle+8:2.5) .. (0,0);
\foreach \angle/\color in {3/black,93/gray} \draw[postaction={decorate},color=\color] (0,0) .. controls +(\angle:4.5) and +(\angle+84:4.5) .. (0,0);
\draw[postaction={decorate},color=black,dotted] (97:0.75) .. controls +(96:1.5) and +(84:1.5) .. (83:0.75);
\draw[postaction={decorate},color=black,dotted] (60:2.4) .. controls +(170:1) and +(10:1) .. (120:2.4);
\draw[decorate,decoration={brace,amplitude=8pt},color=gray,xshift=0pt,yshift=0pt,line width=0.5] (255:2.3) -- (195:2.3) node [gray,midway,xshift=-0.4cm,yshift=-0.4cm] {\footnotesize $m$};
\pgfresetboundingbox \clip (-2.25,-2.5) rectangle (2.25,2.5);} \\
(A) $2n+1 \;\to\; 2n$ petals & & (B) $2n \;\to\; (2m,2(n-m))$ petals
\end{tabular}
\end{center}
\captionsetup {width=0.8\textwidth}
\caption{Manipulations of petal diagrams. (A)~The solid line is a $2n$-petal diagram.  It is obtained from a diagram with $2n+1$ petals by replacing the highest arc, marked by a dotted line, with the big outer petal. (B)~The black and gray solid lines are a petal diagram of a two-component link. It is obtained from the $2n$-petal knot diagram with the two dotted segments, having the same multi-crossing.}
\label{fig-moves}
\end{figure}

The two components have $2m$ and $2(n-m)$ petals, where $m$ depends on which inner petal participates in the smoothing.  Starting from the outer petal in Figure~\ref{fig-moves}A, we traverse the curve until we reach the smoothing, at the inner petal adjacent to the arc at height~$2n$. The inner petals are visited from left to right, and each one corresponds to two entries of~$\pi$. The number of petals in this component is hence half the distance between the locations of $2n$ and $2n+1$ in the permutation~$\pi$. Explicitly, $m=d$ or $n-d$ where $d = \left\lfloor |\pi^{-1}(2n+1) - \pi^{-1}(2n)|/2\right\rfloor$. 

For example, if the heights of the original knot are given by $\pi = (2,6,10,4,9,1,3,11,8,7,5)$ and $\tau = (10\;11)$, then $m=\lfloor |8-3|/2 \rfloor = 2$ and the two resulting components have height sequences $(4,9,1,3)$ and $(8,7,5,2,6,10)$, which concatenate into $\pi_{10}$.

Note that if the locations of $2n$ and $2n+1$ are adjacent in $\pi$, then $d=0$ and the smoothing takes place within the big loop in Figure~\ref{fig-moves}A. This edge case yields a two-component link with $2n$ and $0$ petals, where the $0$-petal component is a disjoint unknot. We remark that this special case can be simplified further to $2n-1$ petals, but we regard it as a $2n$-petal link, to be consistent with the general case.

Suppose that the locations of $2n$ and $2n+1$ in $\pi$ are fixed, while the other entries of $\pi$ are uniformly random among the $(2n-1)!$ possibilities. Then $m$ is also determined, and so is the location of $2n$ in the permutation $\pi_{2n}$ that describes the resulting $(2m,2(n-m))$-petal $2$-component link, but the other $2n-1$ entries of $\pi_{2n}$ are uniformly random. 

We claim that for such $\pi$ the distribution of the smoothed link $L_{2m,2(n-m)}(\pi_{2n})$ is the same as $L_{2m,2(n-m)}(\sigma)$ where $\sigma \in S_{2n}$ is uniform. Indeed, let $\sigma(i) = (\pi_{2n}(i)+j)\bmod 2n$, where $j \in \{1,2,\dots,2n\}$ is uniformly random and independent of $\pi$. Such a rotation preserves the link type as mentioned above, but the resulting permutation becomes uniform in $S_{2n}$.

It follows that if $\pi \in S_{2n+1}$ is uniformly random and we condition on the implied value of $m$, then the smoothed $(2m,2(n-m))$-petal link is distributed exactly as in the Petaluma model.
\end{proof}

The following lemma explores the effect of swapping several pairs of adjacent arcs in a petal diagram on the Casson invariant. We successively apply the relation $\mathrm{c}_2(\tikz[thick,->,baseline=-2]{\draw[black](0.3,-0.1) -- (0,0.2);\draw[white,-,line width=3](0.03,-0.1) -- (0.33,0.2);\draw[black](0.03,-0.1) -- (0.33,0.2);}) - \mathrm{c}_2(\tikz[thick,->,baseline=-2]{\draw[black](0.03,-0.1) -- (0.33,0.2);\draw[white,-,line width=3](0.3,-0.1) -- (0,0.2);\draw[black](0.3,-0.1) -- (0,0.2);}) = \mathrm{lk}(\mathlarger{\mathlarger{\rcurvearrowup\lcurvearrowup}})$, and we have to account for the effect of previous swaps on each linking number.

\begin{lemma} \label{error}
Let $\pi' \in S_{2n+1}$ and let $\pi'' = \tau_1 \circ \dots \circ \tau_k \circ \pi'$, where $\tau_i = (t_i \;\; t_i+1)$ are $k$ disjoint swaps of consecutive numbers. Then
$$ \mathrm{c}_2(K_{2n+1}(\pi'')) - \mathrm{c}_2(K_{2n+1}(\pi')) \;=\; \sum_{i=1}^k \varepsilon(\pi', \tau_i) \, \mathrm{lk}(L(\pi',\tau_i)) \;+\sum_{1 \leq i<j \leq k}\delta(\pi', \tau_i, \tau_j) $$
where $\varepsilon(\pi', \tau_i) = \pm 1 $ and $|\delta(\pi', \tau_i, \tau_j)| \leq 1 $.
\end{lemma}

\begin{proof}
Since the variation of the Casson invariant with respect to a crossing change is the linking number of its smoothing,
$$ \mathrm{c}_2(K_{2n+1}(\tau_1 \circ \pi')) - \mathrm{c}_2(K_{2n+1}(\pi')) \;=\; \pm \mathrm{lk}(L(\pi',\tau_1)) $$
where the sign is determined by the relative orientations of the swapped arcs. Observe that the linking number corresponding to the next swap might depend on whether the current one takes place or not, even though the swaps are disjoint,
$$ \mathrm{lk}(L(\tau_1 \circ \pi',\tau_2)) - \mathrm{lk}(L(\pi',\tau_2)) \;\in\; \{-1,0,1\} \;. $$ 
Indeed, the effect of one crossing change on the linking number of a future smoothing is $\pm 1$ or $0$ depending on whether one or two branches of the smoothing occur at the crossing. Two successive swaps yield
$$ \mathrm{c}_2(K_{2n+1}(\tau_1 \circ \tau_2 \circ \pi')) - \mathrm{c}_2(K_{2n+1}(\pi')) \;=\; \pm \mathrm{lk}(L(\pi',\tau_1)) \pm \mathrm{lk}(L(\pi',\tau_2)) + \delta(\pi', \tau_1, \tau_2) $$
and the general case of $k$ swaps follows by iteration.
\end{proof}

Lemma~\ref{error} will be applied in the proof of Theorem~\ref{c2} with a random permutation and a random set of swaps. Similar to the linking number in Theorem~\ref{lk}, we will show that for almost all permutations, the Casson invariant avoids any particular value for almost all swap sets. In order to track the effect of potential swaps, the terms in the first sum of Lemma~\ref{error} will have to be larger than the second sum. The following lemma will supply us with many such potential swaps with large linking numbers. 

\begin{lemma} \label{swaps}
Let $\pi \in S_{2n+1}$ be uniformly random, and let $k \leq n/8$. Then, the probability that $\left|\mathrm{lk}\left(L\left(\pi, \tau\right)\right)\right| < 2k^2$ for more than $7k$ of the following $8k$ swaps
$$ \tau \in \left\{(1\;2), (3\;4), \dots, (16k{-}1\;16k)\right\}$$ 
is at most $\;3/k\;+\;96k^2/\sqrt{n}$.
\end{lemma}

\begin{proof}
Denote $\tau_i = (2i-1\;\;2i)$ for $i \in \{1,\dots,8k\}$. By Lemma~\ref{smooth}, the two components of $L(\pi,\tau_i)$ have $2m_i$ and $2(n-m_i)$ petals, where either $m_i$ or $n-m_i$ is half the distance $|\pi^{-1}(2i)-\pi^{-1}(2i-1)|$. By Lemma~\ref{cycle} below, applied with $N=2n+1$ and $K=8k$, the probability that $n/4 \leq m_i \leq 3n/4$ for less than $2k$ swaps is at most~$3/k$. We hence proceed assuming at least $2k$ \emph{balanced} links, with the number of petals for each component bounded below by $\min(m_i,n-m_i) \geq n/4$.

For each balanced link, we apply Theorem~\ref{lk} and conclude that the probability of $\text{lk}(L(\pi,\tau_i))$ attaining any particular value is at most $6/\sqrt{n/4} = 12/\sqrt{n}$. Therefore, the probability of having a \emph{small} link, with $|\text{lk}(L(\pi,\tau_i))| < 2k^2$, is at most $48k^2/\sqrt{n}$. 

We need to show that with high enough probability no more than $k$ of these $2k$ links are small. Note that these linking numbers might be strongly correlated. However, Markov's inequality~\cite{ross2009first} guarantees that if each of $2k$ events occurs with probability at most $p$, then the probability that more than $k$ of them occur is at most $2p$. This implies that $|\text{lk}(L(\pi,\tau_i))| < 2k^2$ for more than $k$ of the $2k$ links with probability at most $96k^2/\sqrt{n}$. 

The lemma follows by the union bound on having less than $2k$ balanced smoothed links and having more than $k$ small links. In the complementary case, we have $k$ swaps as desired. 
\end{proof}

\begin{lemma} \label{cycle}
Let $\pi \in S_N$ be uniformly random, and $K \leq N/2$. The following event holds with probability at most $24/K$:
$$ \#\left\{i \in \{1,\dots,K\} \;:\; \frac{N}{4} \leq |\pi^{-1}(2i) - \pi^{-1}(2i-1)| \leq \frac{3N}{4} \right\} \; < \; \frac{K}{4} $$
\end{lemma}

\begin{proof}
Let $Z$ be the quantity counted in the lemma. As in the proof of Lemma~\ref{match}, we can use Chebyshev's inequality for~$Z$. Denote $Z = Z_1 + 
\dots + Z_K$ where $Z_i = 1$ if $i$ is counted and $0$ otherwise. 

Note that the interval $[N/4,3N/4]$ contains between $(N-1)/2$ and $(N+2)/2$ integers, and includes $|\pi^{-1}(2i) - \pi^{-1}(2i-1)|$ with probability at least half. Therefore $E[Z_i] \geq 1/2$ so that $E[Z] \geq K/2$. Trivially $V[Z_i] \leq 1/4$. By similar counting arguments we estimate for $i \neq j$,
$$ COV[Z_i,Z_j] \;\leq\; \frac{(N+2)/2}{N-1} \cdot \frac{(N+2)/2}{N-3} - \left(\frac{1}{2} \right)^2 \;\leq\; \frac{2.5}{N} $$ 
where we used the fact that for $N \leq 48$ the lemma clearly holds. In conclusion,
$$ V[Z] \;\leq\; K \cdot \frac{1}{4} + K^2 \cdot \frac{2.5}{N} \;\leq\; \frac{3K}{2} $$ 
and by Chebyshev's inequality,
$$ P\left[Z \leq \frac{K}{4}\right] \;\leq\; P\left[Z \leq \frac{E[Z]}{2}\right] \;\leq\; \frac{4\,V[Z]}{E[Z]^2} \;\leq\; \frac{24}{K} $$
as required.
\end{proof}

Finally, we prove Theorem~\ref{c2}, establishing $P[\mathrm{c}_2=v] \leq 8/\sqrt[10]{n}$. As mentioned above, the main idea of the proof is swapping certain entries of $\pi \in S_{2n+1}$, such that with high enough probability many of the potential swaps change $\mathrm{c}_2$ significantly. Performing a random subset of such swaps, we use the Littlewood--Offord bound on the probability that these changes add up to a value close to~$v$. If they don't then we show that $\mathrm{c}_2 \neq v$ even after taking into account the error term coming from pairwise dependencies.

\begin{proof}[Proof of Theorem~\ref{c2}]
Let $\pi \in S_{2n+1}$ be uniformly random. Consider the $8k$ swaps $(1\;2), (3\;4)$, $(5\;6), \dots, (16k{-}1\;16k)$ where $k = \lceil \sqrt[5]{n}/8 \rceil$. We modify $\pi$ by a random subset of these swaps, uniformly picked from all $2^{8k}$ subsets. Clearly, the resulting permutation, denoted $\pi''$, is still uniformly random.

To analyze this procedure, it is convenient to perform/ the swaps in a certain order. A swap $\tau$ is called \emph{big} if $|\text{lk}(L(\pi,\tau))| \geq 2k^2$. We perform big swaps \emph{after} the other ones. Lemma~\ref{swaps} shows that only with probability smaller than $3/k+96k^2/\sqrt{n}$ we wouldn't have at least $k$ big swaps.

Denote by $\pi'$ the intermediate permutation after the first $7k$ potential swaps for $\pi$, and before the last $k$ potential big ones $\tau_1,\dots,\tau_k$ that will eventually yield $\pi''$. As we have made at most $7k$ crossing changes since we identified the big swaps, $ |\text{lk}(L(\pi',\tau_i))| \geq |\text{lk}(L(\pi,\tau_i))| - 7k$. Assuming $k \geq 7$, this means $|\text{lk}(L(\pi',\tau_i))| \geq k^2$ for all big swaps. Note that if $k < 7$ then $8/\sqrt[10]{n}$ is larger than one, and the theorem is trivially true.

We apply Lemma~\ref{error} to the last $k$ potential swaps, that yield $\pi''$ from $\pi'$:
$$ \mathrm{c}_2(K_{2n+1}(\pi'')) \;=\; \mathrm{c}_2(K_{2n+1}(\pi')) + \sum_{i=1}^k X_i\varepsilon(\pi', \tau_i) \, \mathrm{lk}(L(\pi',\tau_i)) \;+\sum_{1 \leq i<j \leq k}X_iX_j\delta(\pi', \tau_i, \tau_j) $$
where $X_i=1$ if the $i$-th big swap took place, and $0$ otherwise. Here $\mathrm{c}_2(K_{2n+1}(\pi'))$ is some constant that doesn't depend on the last $k$ swaps. We then apply Theorem~\ref{lo} to the first sum, with $|a_i| \geq k^2$. After adding this sum, $\mathrm{c}_2$ falls in any interval $(v-k^2/2,v+k^2/2)$ with probability smaller than~$1/\sqrt{k}$. The magnitude of the second sum is at most ${k(k-1)}/{2} < k^2/2$. Hence $v$ is still attained with probability at most~$1/\sqrt{k}$.

To conclude, $P\left[\mathrm{c}_2\left(K_{2n+1}\right) = v\right]$ is bounded by the union of two events: having less than $k$ big swaps, and otherwise actually attaining $v$ for the value of $\mathrm{c}_2$ after swapping. With $n \geq 8^{10}$ and $k = \lceil \sqrt[5]{n}/8 \rceil \geq \sqrt[10]{n}$, the probabilities add up to at most
$$ P\left[\mathrm{c}_2\left(K_{2n+1}\right) = v\right] \;\leq\; \frac{3}{k} + \frac{96k^2}{\sqrt{n}} + \frac{1}{\sqrt{k}} 
\;\leq\; \frac{3 \;+\; 96\,/\,7^2 \;+\; \sqrt{8}}{\sqrt[10]{n}} \;\leq\; \frac{8}{\sqrt[10]{n}} $$
as promised.
\end{proof}

\begin{discq*} ~
\begin{enumerate}
\item 
The results in~\cite{even2016invariants} and further numerical experiments~\cite{models} indicate that our upper bound on $P[\mathrm{c}_2(K_{2n+1})=v]$ is not expected to be tight. It remains desirable to establish a bound of $O(n^{-2})$ in Theorem~\ref{c2}. It is plausible that these bounds can be extended to other finite type invariants.
\item
As for the probability mass function $P[K_{2n+1}=K]$, we conjecture that for every $K$ it decays at least exponentially fast in $n$. Even the special case where $K$ is the unknot is interesting. Of course, proving it would require the investigation of more invariants.
\item
Although we couldn't show that $K_{2n+1}$ is non-trivial by finding small summands in its decomposition, we wonder at what probability $K_{2n+1}$ contains, say, a trefoil summand? We can show $\Omega(n^{-3})$ but conjecture it's $o(1)$.

In fact, the above-mentioned experiments by Adams and Kehne~\cite{adams2016bipyramid,uberluma} indicate that $K_{2n+1}$ is prime with high probability. Why is this? 
Note that random knots are not prime in most of the considered random models.
\end{enumerate}
\end{discq*}


{
\footnotesize

\bibliographystyle{alpha}
\bibliography{petaltail}

}

\end{document}